\documentclass[11pt]{amsart}
\usepackage{amsmath, natbib, graphicx, color, latexsym}

\newtheorem{theorem}{Theorem}

\newcommand{\E}{\mathbb{E}}

\newcommand{\note}[1]{{\color{red} ***** #1 *****}}

\begin{document}
\bibliographystyle{plainnat}

\title[Monotonicity of $C(a+\beta\sqrt{a},a)$]%
{A short note on the monotonicity of the Erlang C formula in the Halfin-Whitt regime}

\author[B. D'Auria]{Bernardo D'Auria}
\address{Madrid University Carlos III,
Statistic Department
Avda. Universidad, 30, 28911 Legan\'es (Madird), Spain}
\email{bernardo.dauria.uc3m.es}

\thanks{%
%
This research has been partially supported by the Spanish Ministry of Education and
Science Grants MTM2010-16519, SEJ2007-64500 and RYC-2009-04671.%
}
\renewcommand{\subjclassname}{\textup{2000} Mathematics Subject Classification}

\date{\today}

\begin{abstract} {
We prove a monotonicity condition satisfied by the Erlang C formula when computed in the
Halfin-Whitt regime. This property was recently conjectured in
\cite{janssen:vanleeuwaarden:zwart:2011}.
}
\end{abstract}
\maketitle
Recently, there has been a renewed interest in the Erlang C formula
\begin{equation}
 C(n,a) = 
\frac{a^n}{n!(1-\rho)}
\left(\sum_{k=0}^{n-1} \frac{a^k}{k!} + \frac{a^n}{n!(1-\rho)}\right)^{-1} \ , 
\end{equation}
that gives the probability of waiting for an arriving customers to an $M/M/n$ system
whose traffic intensity is $\rho = a/n$. The value $a$ is the \text{offered load}, that is
the ratio between the arrival rate $\lambda$ and the service rate $\mu$.

The Erlang C formula founds application in the dimensioning of large call servers, i.e. when 
the traffic intensity approaches the instability region ($\rho \to 1$) while
the system keeps high its efficiency by accordingly increasing the number of
servers ($n\to\infty$).
If the scaling is done in an appropriate way, the limiting system still shows a non degenerate
behavior. 

This limit is known as the Halfin-Whitt regime, from \cite{halfin:whitt:1981}, and it requires that
the scaling is done via the \emph{square-root staffing principle}, see also
\cite{borst:mandelbaum:reiman:2004}. The idea is to let the number of servers increase
more than linearly with respect to the offered load $a$ by adding a term proportional to the square
root of $a$, i.e.
\begin{equation}\label{eq:n.fun.a}
n(a) = a + \beta \sqrt{a} \ .
\end{equation}

In the scaling procedure the number of servers is allowed to take not integer values,
hence it is useful to consider the extended version the Erlang C formula to the positive real
numbers, see \cite{jagers:vandoorn:1986},
\begin{equation}\label{C}
 C(s,a) = \left( \int_0^\infty (1+t)^s \frac{a t}{1+t} e^{-a t} \ dt  \right)^{-1} \ ,
\end{equation}
that is valid for $0 < a < s$. 
In \cite{halfin:whitt:1981}, it is shown that
\begin{equation}\label{C.star}
C_*(\beta) 
= \lim_{a\to\infty} C(a + \beta \sqrt{a},a) 
= \left(1+\beta \frac{\Phi(\beta)}{\phi(\beta)}\right)^{-1}
\end{equation}
with $\Phi$ and $\phi$ being respectively the distribution and the density functions of a standard
Normal random variable. 

As observed in \cite{janssen:vanleeuwaarden:zwart:2011}, the value $C_*(\beta) $ in (\ref{C.star})
can be used as a first approximation of $C(a + \beta \sqrt{a},a)$ for large $a$, and eventually as a
dimensioning tool for large call centers as shown in \cite{borst:mandelbaum:reiman:2004}, and
supported by \cite{janssen:vanleeuwaarden:zwart:2009}.

In particular in \cite{janssen:vanleeuwaarden:zwart:2009}, it has been conjectured that 
for any value of $\beta$ the function $C(a + \beta \sqrt{a},a)$ is decreasing, that is the limit
in (\ref{C.star}) is approached from above and in particular the value of $C_*(\beta)$ is a 
performance lower bound for any system with offered load $a<1$.

The purpose of this note is to prove that indeed this conjecture holds true. The main tool of the
proof is in realizing that the function $1/C(a + \beta \sqrt{a},a)$ can be written in
term of the moments of some special random variables that are proved to be stochastically ordered.
It is interesting to note that even if the role of these random variables is crucial in the proof we
were not able to give a \emph{physical} interpretation to them. Probably if this could be done it
would add additional insights to the structure of the function $C(a + \beta \sqrt{a},a)$ and in
general to the asymptotic queueing system in the Halfin-Whitt regime.

\begin{theorem}
The function $C(a+ \beta \sqrt{a},a)$ is strictly decreasing in $a$ for any fixed $\beta>0$.
In particular, for any $a>0$, $C(a + \beta \sqrt{a},a) > C_*(\beta)$.
\end{theorem}
\begin{proof}
Consider the function  $g(t,a)=a t\,e^{-a \, t} \ (1+t)^{a- 1}$, with $a$, $t>0$.
It is non negative and $\int_0^\infty g(t,a) \ dt = 1$ so we can look at it as a density function of
a non negative random variable, say $X_a$.

Let $Y_a = (1+X_a)^{\sqrt{a}}$, we can express the Erlang C function, see (\ref{C}), in terms of
the inverse of the $\beta$ moments of the random variables $Y_a$, indeed
$$C(a+\beta\,\sqrt{a},a) = 1/\E[(1+X_a)^{\beta\sqrt{a}}] = 1/\E[Y_a ^\beta] \ ,$$
and we need to prove that the $\beta$ moments, with $\beta>0$ fixed, are strictly increasing in $a$.
Since the power function $(\cdot)^\beta$ is increasing, it is enough to prove that the family of
random variables $\{Y_a\}_{a>0}$ is increasingly stochastically ordered.

First we compute their density functions. Let $y(x) = (1+x)^{\sqrt{a}}$, we have that the inverse
$x(y)=y^{\frac{1}{\sqrt{a}}} -1$, with $x>0$ and $y>1$ and  $x'(y) =\frac{1}{\sqrt{a}}
y^{\frac{1}{\sqrt{a}}-1}$, that is positive for any $y>1$. 
It follows that the density function $f(t,a)$ of $Y_a$ is given by
$$f(y,a) = g(x(y),a) \ x'(y) = \sqrt{a} \  y^{\sqrt{a}-1} (y^{\frac{1}{\sqrt{a}}} - 1) e^{- a
(y^{\frac{1}{\sqrt{a}}}-1)} \ ,$$
while its tail distribution has the following expression
\begin{equation}\label{y.tail}
\bar F(y,a) = \Pr\{Y_a > y \} = \int_y^\infty f(t,a) \ dt = y^{\sqrt{a}} \
e^{-a(y^{\frac{1}{\sqrt{a}}}-1)}
\end{equation}
for $y>1$. An easy check that $Y_a$, with $a>0$, is indeed a random variable can be done by differentiating equation $(\ref{y.tail})$ with respect to $y$ and getting $-f(y,a)$. Having that $f(y,a)$ is non negative for $y \geq 1$ it follows that $\bar F(y,a)$ is a decreasing function in $y$, and the check is complete by noticing that $\bar F(y,a) \to 0$ as $y \to \infty$ and $\bar F(1,a)=1$. To prove that the stochastic order holds it is sufficient to show that the function $\bar F(y,a)$ is increasing in $a>0$. 

We can rewrite the tail distribution in the following way
\begin{eqnarray}\label{tF.2}
\bar F(y,a) 
&=& \exp\{(\log{y})^2 \, \frac{\sqrt{a}}{\log{y}}
	-(\log{y})^2 \, (\frac{\sqrt{a}}{\log{y}})^2(e^{\frac{\log{y}}{\sqrt{a}}}-1)\} \\
&=& \exp\{(\log{y})^2 \, h(\frac{\sqrt{a}}{\log{y}})\} \nonumber \,
\end{eqnarray}
with $h(x)=x+x^2 (1-e^{1/x})$. Since the exponential and the square root functions are increasing functions in their respective domains, we only need to show that $h(x)$ is an increasing function for $x>0$. Representing the exponential function by its Taylor series we have
\begin{equation}\label{eq:U}
h(x)
= x+x^2 \Big(1- \sum_{n=0}^\infty \frac{x^{-n}}{n!}\Big) 
= x - \sum_{n=1}^\infty \frac{x^{2-n}}{n!} 
= - \sum_{n=0}^\infty \frac{x^{-n}}{(n+2)!}
\end{equation}
and the result follows because it is a sum of increasing functions.
\end{proof}
As a final remark, we mention that the Halfin-Whitt regime can be obtained in another way as well. That is, instead of writing the number of servers as a function of the offered load, like in (\ref{eq:n.fun.a}), the latter is written as a function of the former in the following way
\begin{equation}\label{eq:a.fun.n}
a(n) = n - \beta \sqrt{n}
\end{equation}
where in general $\beta>0$, and the relation is valid for $n>\beta^2$.

A natural question is to ask if the monotonicity property satisfied by $C(s(a),a)$, $a>0$ is also valid for
$C(s,a(s))$, with  $s>\beta^2$. However the behaviour of the $C(s,a(s))$ looks more complicated as the following graphs show for two values of the parameter, $\beta=1/10$ and $\beta=3$.

\begin{minipage}{0.5\textwidth}\centering
\includegraphics[width=0.9\textwidth]{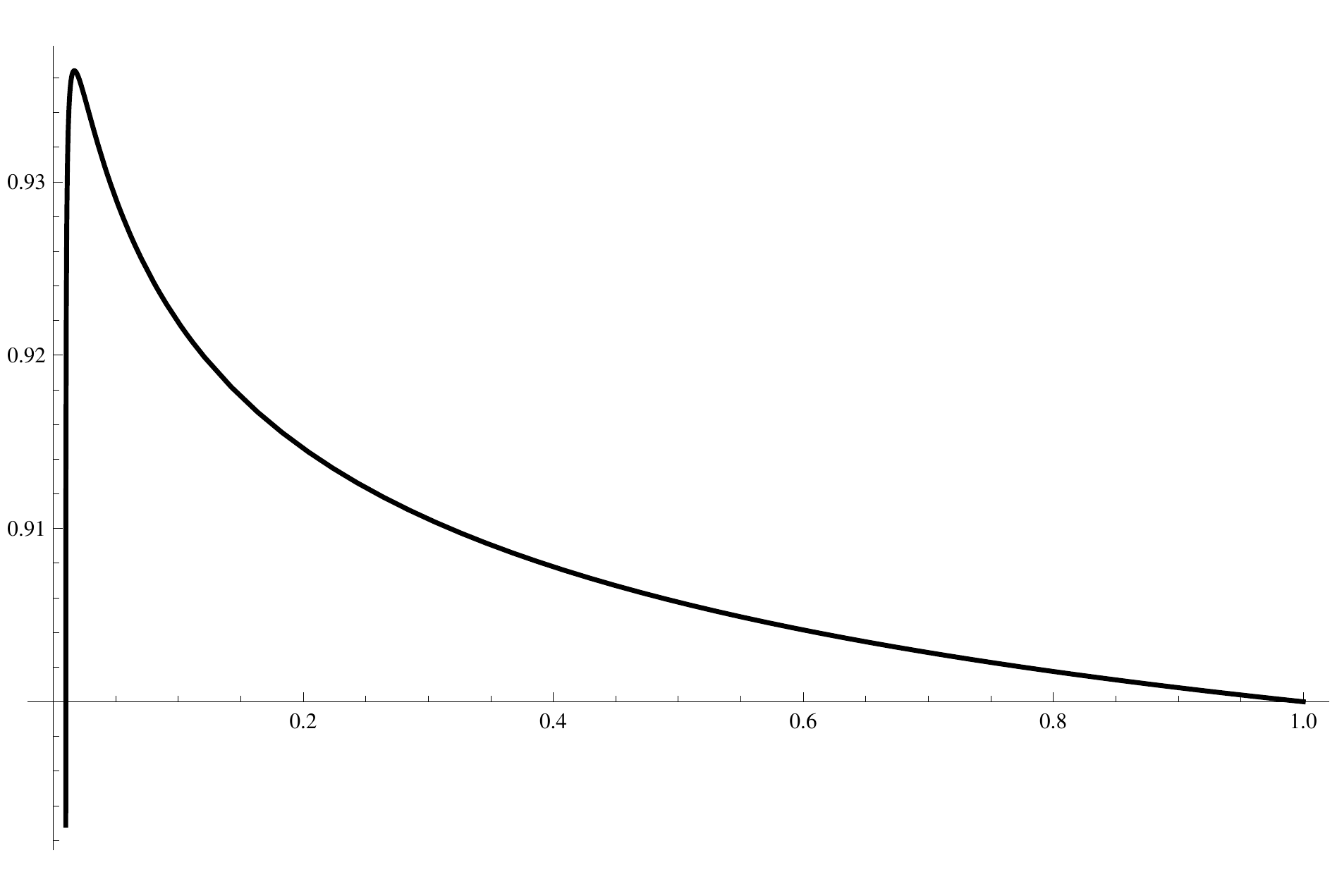}
$$C(s,s-\frac{1}{10}\sqrt{s})$$
\end{minipage}
\begin{minipage}{0.5\textwidth}\centering
\includegraphics[width=0.9\textwidth]{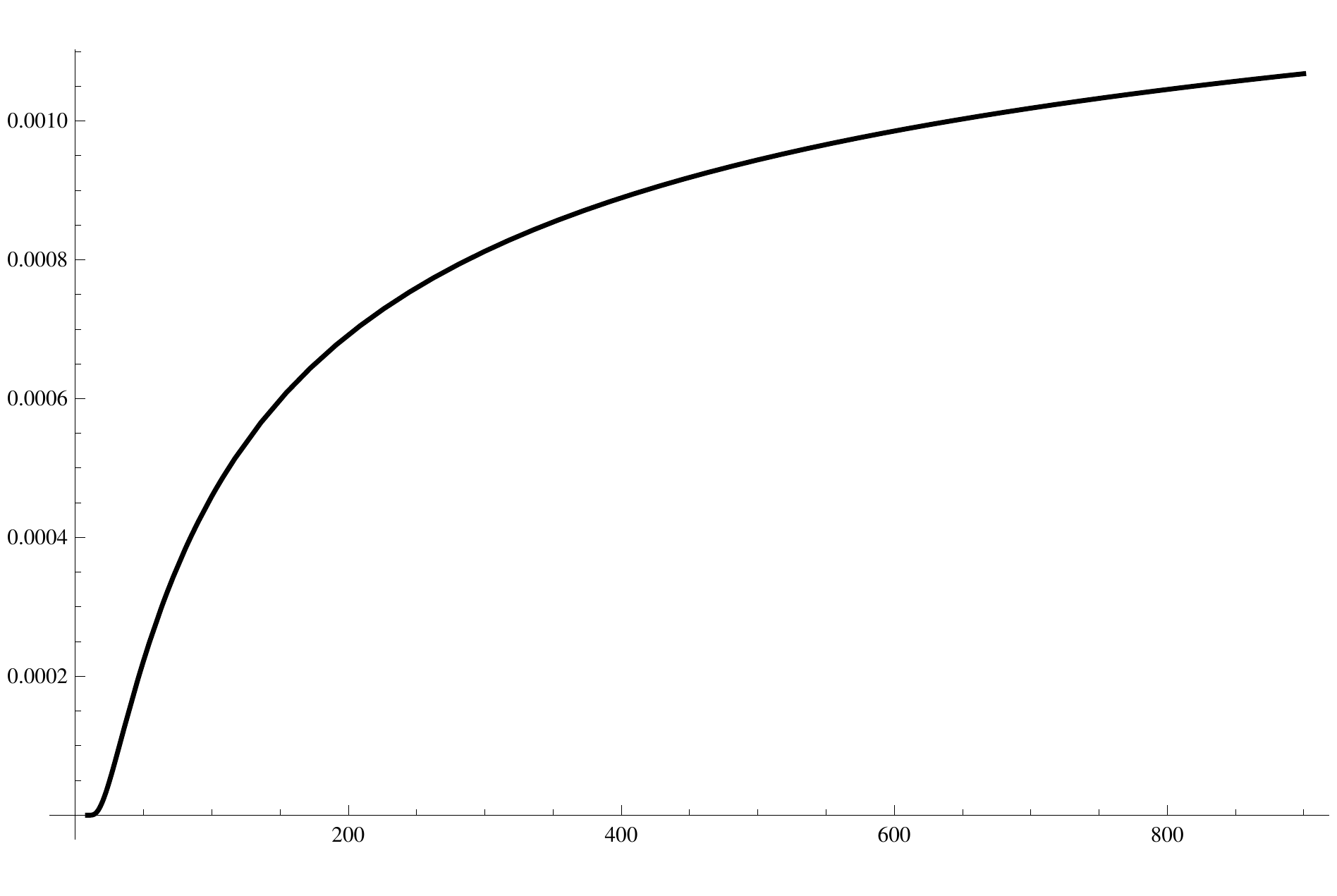}
$$C(s,s-3\sqrt{s})$$
\end{minipage}


\end{document}